\documentclass[a4paper,12pt]{article}
\usepackage{hyperref}
\usepackage{amsmath}
\usepackage{amssymb}
\usepackage{amsthm}
\usepackage{mathrsfs}
\newtheorem{theorem}{Theorem}[section]
\newtheorem{lemma}[theorem]{Lemma}

\newtheorem{definition}[theorem]{Definition}
\newtheorem{corollary}[theorem]{Corollary}
\newtheorem{remark}[theorem]{Remark}

\newcommand{\EE}{\mathbb{E}}
\newcommand{\FF}{\mathbb{F}}
\newcommand{\NN}{\mathbb{N}}

\newcommand{\PP}{\mathbb{P}}

\newcommand{\AS}{\mathbb{S}}

\newcommand{\DD}{\mathbb{D}}

\newcommand{\B}{{\cal B}}

\newcommand{\D}{{\cal D}}
\newcommand{\E}{{\cal E}}

\newcommand{\J}{{\cal J}}

\newcommand{\tT}{{\cal T}}

\newcommand{\Y}{{\cal Y}}
\newcommand{\cZ}{{\cal Z}}

\begin{document}

\title{\itshape A probabilistic analysis of a discrete-time 
evolution in recombination}
\author{Servet Mart{\'i}nez}

\maketitle

\begin{abstract}
We give a closed form of the discrete-time evolution of 
a recombination transformation in population genetics. 
This decomposition allows to define a Markov chain
in a natural way. 
We describe the geometric decay rate to the limit 
distribution,
and the quasi-stationary behavior 
when conditioned to the event that the chain 
does not hit the limit distribution.
\end{abstract}

\bigskip

\noindent {\bf Keywords: $\,$} Markov chain; Population 
genetics;   Recombination; geometric decay rate; quasi-stationary
distributions.  

\bigskip

\noindent {\bf AMS Subject Classification:\,} 60J10; 92D10. 

\section{ Introduction }
\label{sec0}

Here we study the evolution of the following transformation 
$\Xi$ acting on the set of probability 
measures $\mu$ on a product measurable space $\prod_{i\in I}A_i$, 
$$
\Xi[\mu]=\sum_{J\subseteq I} \rho_J \, \mu_J \otimes \mu_{J^c}.
$$
Here $\rho=(\rho_J: J\subseteq I)$ is a probability vector,
$\mu_J$ and $\mu_{J^c}$ are the marginals of $\mu$ on 
$\prod_{i\in J}A_i$ and $\prod_{i\in J^c}A_i$ respectively, and 
$\otimes$ means that these marginals 
are combined in an independent way.

\medskip

The analysis of $\Xi$ should give an insight in the  
study of the genetic composition 
of population under recombination. Genetic information is encoded 
in terms of sequences of symbols indexed by a finite set of 
sites. In the process of 
recombination the children sequences are derived from two parents, 
a subset of sites is encoded with the 
maternal symbols and the complementary set 
is encoded with the paternal symbols. 
The above equation expresses 
that these sets $(J,J^c)$ constitute a
probabilistic object distributed according to $\rho$. A relevant 
feature is that recombination produces decorrelation between 
sites and this is expressed 
by the fact that the sequence distribution 
on these sets are grouped independently into $\Xi[\mu]$.

\medskip

The evolution $(\Xi^n[\mu])$ has been mainly studied in the
context of single cross-overs, that is where $I=\{1,..,K\}$ and 
the pairs of sets $(J,J^c)$ are of the form
$J=\{i: i<j\}$, $J^c=\{i: i\ge j\}$. 
This evolution was introduced by H. Geiringer \cite{ge}, and firstly 
solved in the continuous-time case by E. Baake and M. Baake
\cite{bb}, where it is also supplied an important corpus of ideas 
and techniques to study the discrete-time evolution. 
In relation to the discrete-time evolution 
we refer to \cite{bvw}: 
{\it '...the corresponding discrete-time dynamics, which is prevalent 
in the biological literature, is more difficult: its solution has, 
so far, required nontrivial
transformations and recursions that have not been solved in closed 
form (Benett 1954; Dawson 2000, 2002; von Wangenheim et al. 2010).'}
These last works are cited in our list of references as \cite{be}; 
\cite{daw1}, \cite{daw2}; \cite{vwbb}.

\medskip

Richer discussions on the interpretation of the above equation
in a broader perspective of recombination in population genetics, 
are given in the 
introductory sections of references \cite{bb}, \cite{bvw}, 
\cite{bbs} and \cite{uvw}. 

\medskip

When studying single cross-over recombination, one the main
objectives in \cite{uvw} and \cite{bvw} is
to express the iterated $\Xi^n[\mu]$ in a simple form. The main 
tools in these works are M\"{o}bius inversion formulae, 
similarly to the continuous case, 
and commutation
relations between $\Xi$ and recombination operators. Some of the 
main results of these works are the one step recursive relation 
stated in Theorem 1 in \cite{bvw},  
Proposition 3.3 in \cite{uvw} stating that  
if one starts from a distribution
$\mu$ then $\Xi^n[\mu]$ converges to the Bernoulli distribution
having the marginals of $\mu$, and the relation to ancestry trees
and Markov chains, summarized in Theorem 3 in \cite{bvw}.

\medskip

In our work we present two main results, these are Theorems
\ref{theo0} and \ref{theo1}. 

\medskip

In Theorem \ref{theo0} we write $\Xi^n[\mu]$ as a weighted decomposition 
of $\otimes_{\ell \in \delta} \mu_\ell$, where
$\mu_\ell$ is the marginal $\mu$ on the set $\ell$, and 
$\{\ell \in \delta\}$ are the 
atoms of some partition $\delta$ of $I$, and we give exactly the weights 
of this decomposition. This follows from a simple backward 
development of $\Xi^n[\mu]$ done in Lemma \ref{lemmafd}. 
When looking in detail the formulae stated in Theorem \ref{theo0} 
one realizes that
they define a natural Markov chain $(Y_n)$ on the set 
of partitions of $I$, having the remarkable property that when 
it starts from the coarsest partition $\{I\}$, then the probability 
that $\{Y_n=\delta\}$ is equal to the sum 
of weights of all trees participating in the backward development of
$\Xi^n[\mu]$ whose set of leaves is $\{\ell \in \delta\}$.
These results are Lemmas \ref{lemma4} and \ref{lemma5}.

\medskip

In Theorem \ref{theo1} we use this Markov chain to describe the 
geometric coefficient of convergence to the limit distribution
$\otimes_{\ell \in \D^\rho} \mu_\ell$, where
$\D^\rho$ is the partition generated by the sets $\{J: \rho_J>0\}$.
In the single cross-over case the atoms of this partition
are the singletons, so the limit probability measure is the
Bernoulli distribution. 
A key result is formula (\ref{50e}) that characterizes 
the geometric decay behavior. 
In this Theorem we also study in detail the limiting conditional 
behavior of the chain when conditioned to the fact that it has not hit 
the limit distribution. 
Besides giving the limiting conditional distribution, we state
a ratio limit of the probabilities of not hitting the limit
distribution. We emphasize that these last results are not 
a consequence of any known result in 
the theory of quasi-stationary distributions because the Markov 
chain $(Y_n)$ is not irreducible on the class of non-absorbing states,  
so we are not able to use the Perron-Frobenius theory. All these results
require entirely new computations. Quasi-stationary distributions
have been studied mostly in relation to population extinction,
see for instance Section 2.6 in \cite{les}, and \cite{pp, cms} 
for a wide ranging bibliography on the subject. 
In our context the absorbing state is not the void population as 
happens in extinction, and a main interest 
of the quasi-limiting behavior is in the process that never hit
the limit distribution which is given in Corollary \ref{cor1}. 

\medskip

In Section \ref{sec1} we fix notation on partitions, atoms, and 
dyadic partitions. In Section \ref{sec2} we define supply 
some technical lemmas on the transformation $\Xi$. Thus, 
in Lemma \ref{lemma3} we get the marginal $\Xi[\mu]_K$ for $K$ 
a union of atoms, in terms of some iterated coefficients $\rho^K_M$ 
derived from $\rho$ which constitute key quantities along all our 
study. In Section \ref{sec3} we introduce the
dyadic family of trees depending on the support of $\rho$ 
participating in the tree decomposition of $\Xi^n[\mu]$. 
Finally, in Section \ref{sec4}
we introduce the Markov chain on partitions 
and state our main results on the quasi-limiting behavior. 

\medskip

Let us discuss briefly the relations of our results with respect to 
previous literature mainly with respect to \cite{uvw} and \cite{bvw}, 
which have been an important inspiration for our work. In these 
references a Markov chain on partitions was introduced for single 
cross-over recombination, 
by following the ancestry of the genetic material
of a selected individual from a population and using some
limits arguments. As a consequence of this rather 
complicated construction, a key 
relation between the Markov chain and
the coefficients of the iterated $\Xi^n[\mu]$ is stated in Theorem 3 in 
\cite{bvw}, which must be the same relation we
state in Lemma \ref{lemma5}. We note, that each backward step 
in ancestry involves a probabilistic object because 
the dyadic partition $(J,J^c)$ is randomly distributed.
But, our approach differs with the one used in \cite{bvw} at some 
substantial points: we get a closed form 
of $\Xi^n[\mu]$ by using a simple backward decomposition and this 
decomposition suggests the definition of the 
Markov chain $(Y_n)$ in a very natural way. Our techniques are totally
different to those used in \cite{uvw} and \cite{bvw}.
Also, our result apply to all kind
of dyadic partitions $(J,J^c)$ that can have a complex combinatorics
and not only for the ones arising in the single cross-over case.
Finally, the study of the quasi-stationary behavior 
of this chain is, to our knowledge, firstly studied in this 
monograph. 

\medskip

We point out that even if our results are stated for a 
product of finite spaces, they can be stated for general 
product of measurable spaces as 
pointed out in Remark \ref{rem2a}.

\medskip

Recently, in \cite{bbs}, the continuous-time evolution 
was studied in a framework of
general partitions other than dyadic partitions. 
The extension of our results to the analogous framework 
but for discrete-time, deserves a different study.

\medskip

It is worth mentioning, that in Section 
\ref{sec2} and in the final comment of this work, we 
point out that all our results remain true when 
$\otimes$ is a commutative and associative operation, it has 
an identity element and is also stable under 
restriction. It could be explored the existence of good candidates 
for operations $\otimes$ other than the product of probability 
measures, that would be meaningful in population genetics.

\section{ Partitions }
\label{sec1}

In this section we fix some notation on partitions. Some emphasis
is put in defining a partition from a family of sets, with a 
special care in defining the atoms, and we make the difference 
between dyadic and strictly dyadic partitions, the last ones 
having exactly two atoms.

\medskip

Let $I$ be a finite set and
$\AS(I)=\{L: L\subseteq I\}$ be the class of its subsets
($\subseteq$ means inclusion and
$\subset$ strict inclusion).
For any class of sets $\cZ\subseteq \AS(I)$ we put  
$\cZ^{(\emptyset)}=\cZ\setminus \{\emptyset\}$, 
$\cZ^{(I)}=\cZ\setminus \{I\}$ and 
$\cZ^{(\emptyset,I)}=\cZ\setminus \{\emptyset, I\}$.
So, when $\cZ$ does not contain the empty set we 
have $\cZ^{(I)}=\cZ^{(\emptyset,I)}$.

\medskip

A partition $\D$ of $I$ is a collection of nonempty sets
(so $\D\subseteq \AS(I)^{(\emptyset)})$,
pairwise disjoint and covering $I$. 
We note $\D=\{L: L\in \D\}$ and
any of the sets $L$ is called an atom 
of $\D$. We note by $\DD(I)$ the family of partitions of $I$.

\medskip

For $\D, \D'\in \DD(I)$, $\D'$ is said to be finer than 
$\D$ or $\D$ is coarser than $\D'$, 
if every atom of $\D'$ is contained in an atom of $\D$. 
In this case every atom in $\D$ is union of atoms of $\D'$. 
The finer partition is the class 
of singletons $\D_{I,si}=\{\{i\}: i\in I\}$, and the coarsest 
one is $\{I\}$. 

\medskip

Let $\J\subseteq \AS(I)^{(\emptyset)}$ be a nonempty 
family of nonempty sets which satisfies,
$J\in \J, J\neq I \, \Rightarrow \, J^c\in \J$.
Then, it defines a partition $\D(\J)\in \DD(I)$ as follows.
Let $\Y_1(\J)=\J$ and define by recursion
the following family of classes of nonempty sets,
$$
\forall\, n\ge 1: \quad
\Y_{n+1}(\J)=\{K\cap J: K\in \Y_n(\J), J\in \J, 
K\cap J\neq \emptyset\}. 
$$ 
Since
$J\cap J=J$, we have
$\Y_n(\J)\subseteq \Y_{n+1}(\J)$ for all $n\ge 1$.
Also it stabilizes in a finite number of steps, 
that is there exists $n_0\ge 1$ such 
that 
$\Y_{n_0+k}(\J)=\Y_{n_0}(\J)$ for all $k\ge 0$. Let
\begin{equation}
\label{e7}
\Y(\J)=\bigcup_{n\ge 1} \Y_n(\J).
\end{equation}
We define the atoms of the partition $\D(\J)$ by:
\begin{equation}
\label{e8}
L\in \D(\J)\Leftrightarrow \big[ \, L\in \Y(\J) \hbox{ and }
\forall J\in \J : \, J\cap L =L \vee J\cap L =\emptyset \, \big].
\end{equation}
It is clear that the atoms $L\in \D(\J)$ are disjoint, on
the other hand they cover $I$, because in the contrary 
the set $I\setminus (\bigcup_{L\in \D_I(\J)} L)$ would have 
a nonempty intersection with some $J\in \J$ and so it would
contain an atom of the form (\ref{e8}) leading to a 
contradiction. Then,
$I=\bigcup_{L\in \D_I(\J)} L$. It is also straightforward 
to show that
\begin{equation}
\label{e3}
\forall K\in \Y(\J): \quad K=\!\!\!\bigcup_{L\in \D(\J):
L\subseteq K}\!\!\!L.
\end{equation}

\medskip

It is useful to introduce {\it dyadic} partitions.
The set of dyadic partitions on $I$ is noted by
$\DD_{1,2}(I)$ and it is given by
$$
\DD_{1,2}(I)=\{I\}\cup \{\{J,J^c\}: J\in \AS^{\emptyset,I}(I)\}.
$$  
The ${\;}_{1,2}$ subscript
is because a partition $\delta\in \DD_{1,2}(I)$ can have one
or two atoms. It has one atom only in the case 
$\delta=\{I\}$, in all other cases it contains 
two atoms. We will make
the distinction with respect to the family of {\it strictly dyadic} 
partitions, which is the class of partitions having exactly 
two atoms,
$$
\DD_{2}(I)=\{\{J,J^c\}: J\in \AS^{\emptyset, I}(I)\}.
$$

From now on, we fix $I$ and call it the set of sites.
In the notation of the variables we will often delete the
dependence on $I$, thus we write
$\AS=\AS(I)$, $\DD=\DD(I)$, $\DD_{1,2}=\DD_{1,2}(I)$, 
$\DD_2=\DD_2(I)$ and 
so on. But we keep the
dependence of these quantities on a set $J$ when it 
is not necessarily $I$, in this case we write 
$\AS(J)$, $\DD(J)$, $\DD_{1,2}(J)$, $\DD_2(J)$ and so on.

\section{The recombination transformation}
\label{sec2}

In this section we define the action of $\Xi$ on the set of 
probability of measures of a product measurable space. 
For simplicity we assume the product space is finite, in fact 
all our results remain true for a product 
of general measurable spaces. But, the finiteness of the 
set of sites $I$ is crucial. We will 
supply some elementary properties of $\Xi$,
a main one being the description of the marginal of the
transformed probability measure, this is done 
in Lemma \ref{lemma3}. This description is written in terms 
of some coefficients whose main properties are summarized in 
Lemma \ref{lemma2}. We devote some time to state 
exactly the properties of $\otimes$ that will used in this work.

\medskip

Let $A_i$ be a finite set for $i\in I$, called the 
alphabet on site $i$.
Let $\prod_{i\in I}A_i$ be the product space.
In order that our statements are for non-trivial, 
we will assume that the sets $I$ and $A_i$ for $i=1,..,n$,  
contain at least two elements. 

\medskip

We note by $x$ an element of $\prod_{i\in I}A_i$,
so $x=(x_i\in A_i: i\in I)$. 
Denote by ${\cal P}_I$ the set of probability measures
on $\prod_{i\in I}A_i$. Any $\mu\in {\cal P}_I$
is determined by the values
$(\mu(x): x\in \prod_{i\in I}A_i)$.
Let $J\in \AS$. We note $x_J=(x_i: i\in J)$ and
make the identification $x=(x_J,x_{J^c})$. We
denote by ${\cal P}_J$ the set of probability measures
on $\prod_{i\in J}A_i$.

\medskip

The marginal $\mu_J\in {\cal P}_J$ of $\mu\in {\cal P}_I$ 
on $J$ is, 
\begin{equation}
\label{eab2}
\forall x_J\in \prod_{i\in J}A_i: \quad \mu_J(x_J)
:=\mu(\{y\in \prod_{i\in I}A_i: y_i=x_i\})=
\sum_{x_{J^c}\in \prod_{i\in J^c}A_i}\!\!\!\mu(x).
\end{equation}
For $J=I$ we have $\mu_I=\mu$, and for 
$J=\emptyset$ we have $\mu_\emptyset(x_\emptyset)=1$.
We put $\mu_\emptyset\equiv 1$ to get
consistency in all the relations where it will appear.

\medskip
 
If $K\subseteq J$ then the marginal $\mu_{K}$ can be defined
from $\mu_J$, that is it satisfies
\begin{equation}
\label{eab3}
\mu_K(x_K)=\sum_{x_{J\setminus K}\in 
\prod_{i\in J\setminus K}A_i}\mu_J(x_J).
\end{equation}

We take $\otimes$ to be the product measure: for all
$\mu^J\in {\cal P}_J$, $\mu^K\in {\cal P}_K$,
$$
\forall x_{J\cup K}\in \prod_{i\in J\cup K} A_i: 
\quad \mu^J\otimes \mu^K(x_{J\cup K})=\mu^J(x_J)\mu^K(x_K).
$$
Let us explicit the properties we will use from $\otimes$. 
First, the operation $\otimes$ is defined in the domains
\begin{equation}                  
\label{eab0} 
\forall J,K\in \AS, J\cap K=\emptyset; \quad
\otimes: {\cal P}_J\times {\cal P}_K\to {\cal P}_{J\cup K}.
\end{equation}
The operation $\otimes$ is commutative and associative, that is
for $J,K,M\in \AS$ pairwise disjoint,
$\mu^J\in {\cal P}_J$, $\mu^K\in {\cal P}_K$,
$\mu^M\in {\cal P}_M$, it is satisfied
\begin{equation}
\label{eab4}
\mu^J\otimes \mu^K=\mu^K\otimes \mu^J \hbox{ and }
(\mu^J\otimes 
\mu^K)\otimes\mu^M=\mu^J\otimes (\mu^K\otimes\mu^M).
\end{equation}  
Moreover, $\mu^\emptyset\equiv 1$ is an
identity element for $\otimes$, and 
$\otimes$ satisfies the following stability
property under restriction.

\begin{lemma}
\label{lemma0}
For all $J, K, M\in \AS$ with $J\cap K=\emptyset$
and $M\subseteq J\cup K$,
\begin{equation}
\label{eab}
(\mu_J\otimes \mu_{K})_M=\mu_{J\cap M}\otimes \mu_{K\cap M}.                
\end{equation}
\end{lemma}

\begin{proof}
This is a consequence of definition (\ref{eab2}) and
property (\ref{eab3}). In fact
\begin{eqnarray*}
(\mu_J\otimes \mu_{K})_M(x_{(J\cup K)\cap M})
&=&
\sum_{x_{(J\cup K)\setminus M}} 
(\mu_J \otimes \mu_K)(x_{J\cup K})\\
&=&\sum_{(x_{J\setminus M}, x_{K\setminus M})} 
\mu_J(x_J) \cdot \mu_K(x_K)\\ 
&=&\left(\sum_{x_{J\setminus M}}\mu_J(x_J)\right) 
\left(\sum_{x_{K\setminus M}}\mu_K(x_K)\right)\\
&=&\mu_{J\cap M}(x_{J\cap M})\cdot \mu_{K\cap M}(x_{K\cap M})\\
&=&
(\mu_{J\cap M}\otimes \mu_{K\cap M}) (x_{(J\cap M)\cup (K\cap M)}).
\end{eqnarray*}
\end{proof}
The above properties (\ref{eab0}), (\ref{eab4}), (\ref{eab}),
and $\mu^\emptyset$ an identity, are all we need from $\otimes$
to get the results of this work.

\medskip 

Note that commutation and associativity imply that 
$\bigotimes_{L\in \D}\mu_L\in {\cal P}_I$ is well-defined
for a partition $\D$ of $I$. 

\medskip

For $\otimes$ the product measure we have
$$
\forall x\in \prod_{i\in I}A_i: \quad 
\bigotimes_{L\in \D}\mu_L(x)=\prod_{L\in 
\D}\mu_L(x_L).
$$
If $\D=\D^{si}=\{\{i\}: i\in I\}$, 
$\bigotimes_{i\in I} \mu_{\{i\}}$ is 
called Bernoulli and $(\mu^{\{i\}}: i\in I)$ are the
one-site marginals.

\medskip

From now on, we fix $\rho=(\rho_J: J\in \AS)$ 
a probability vector, so $\rho_J\ge 0$ for $J\in \AS$ 
and $\sum_{J\in \AS}\rho_J=1$, and that also 
satisfies $\rho_\emptyset=0$.

\begin{definition}
\label{def1}
Define the following transformation $\Xi: {\cal P}_I\to {\cal P}_I$,
\begin{equation}
\label{e4}
\Xi[\mu]= \sum_{J\in \AS} 
\rho_J \, \mu_J\otimes \mu_{J^c}=\sum_{J\in \AS^{(\emptyset)}}
\rho_J \, \mu_J\otimes \mu_{J^c}. 
\end{equation}
$\Box$
\end{definition}
Since $\mu_\emptyset\equiv 1$, then
\begin{equation}
\label{e4x}
\Xi[\mu]= \rho_I \, \mu + \sum_{J\in \AS^{\emptyset,I}}
\rho_J \, \mu_J\otimes \mu_{J^c}.
\end{equation}

\smallskip

\begin{remark}
\label{rem0ax}
We have $\rho_\emptyset=0$, but 
we can have $\rho_I>0$. On the other hand, 
since $\mu_J\otimes \mu_{J^c}=\mu_{J^c}\otimes \mu_J$, 
we can assume when it is needed that $\rho_J=\rho_{J^c}$ for
$J\in \AS^{(\emptyset,I)}$. 
\end{remark}

Observe that formula (\ref{e4}) can be written in terms of dyadic 
and strictly dyadic partitions as,
$$
\Xi[\mu]=\sum_{\D\in \DD_{1,2}} 
(\sum_{K\in \D}\rho_K)\bigotimes_{K\in \D} \mu_K
=\rho_I \mu+ \sum_{\D\in \DD_{2}}
(\sum_{K\in \D}\rho_K)\bigotimes_{K\in \D} \mu_K.
$$

Let 
$$
\J_\rho=\{J\in \AS^{(\emptyset)}: \rho_J>0\} 
$$
be the support of $\rho$, so 
$\J_\rho^{(I)}=\{J\in \AS^{(\emptyset,I)}: \rho_J>0\}$ is the class 
of nonempty subsets strictly contained in $I$ 
which are in the support of $\rho$. Denote by
$$
D^\rho:=\D(\J_\rho)
$$ 
the partition generated by the class of sets $\J_\rho$, 
whose atoms satisfy (\ref{e8}) with $\J=\J_\rho$. 

\begin{lemma}
\label{lemma1}
$(i)$ Let $M$ be contained in an atom in $\D^\rho$, then $\Xi$ 
preserves the marginal on $M$, that is  
\begin{equation}
\label{e6}
[M\subseteq L, L\in \D^\rho]\, \Rightarrow \, \Xi[\mu]_M=\mu_M.
\end{equation}

\noindent $(ii)$ Let $\D$ be a partition finer than $\D^\rho$
(so $\D=\D^\rho$ is allowed), and
$\mu^{L}\in {\cal P}_L$ 
for $L\in \D$. Then, $\mu=\bigotimes_{L\in  \D} \mu^{L}$
is a fixed point for $\Xi$, that is $\, \Xi[\mu]=\mu$.
\end{lemma}

\begin{proof}
$(i)$ From (\ref{eab}) we have
$(\mu_J\otimes \mu_{J^c})_M=\mu_{J\cap M}\otimes \mu_{J^c\cap M}$.
From (\ref{e8}) and since $M\subseteq L\in \D^\rho$, we get that 
every $J\in \J_\rho$ satisfies:
$J^c\subseteq M^c$ or $J\subseteq M^c$. Then
$(\mu_J\otimes \mu_{J^c})_M=\mu_M$. 

\medskip

\noindent $(ii)$ For every $L\in \D$ 
we have  $\mu_{L}=\mu^{L}$ for all $L\in \D$. From (\ref{e3}) we get, 
$$
\forall J\in \J_\rho:\quad 
\mu_J=\bigotimes_{L\in \D: L\subseteq J}\mu^{L}. 
$$
Since the sets in the families $\{L\in \D: L\subseteq J\}$
and $\{L\in \D: L\subseteq J^c\}$,
are disjoint and their union is $I$, we get
$\mu_J\otimes \mu_{J^c}=\bigotimes_{L\in \D}\mu^{L}=\mu$.
Hence $\Xi[\mu]=\mu$.
\end{proof}

As consequence of Lemma \ref{lemma1} $(ii)$ we get that 
the one-site marginals $\mu_{\{i\}}$ are preserved
by $\Xi$, and the Bernoulli probability measures are fixed
points of $\Xi$.

\medskip

When $\rho_I=1$, so $\Xi[\mu]=\mu$
and $\Xi$ is the identity transformation. Then, in the sequel 
we assume 
$$
\rho_I<1 \hbox{ or equivalently } 
\J_\rho^{(I)}\neq  \emptyset.
$$ 

We recall the notation in (\ref{e7}), $\Y(\J_\rho)$ is 
the class of all nonempty intersections of sets in $\J_\rho$.
For $K\in \Y(\J_\rho)$ we define,
$$
K\cap \J_\rho=\{K\cap J: J\in \J\}. 
$$
We have 
$K\in K\cap \J_\rho$ and 
$K\cap \J_\rho \subseteq \Y(\J_\rho)\cup \{\emptyset\}$.

\begin{definition}
\label{def2}
For all $K\in \Y(\J_\rho)$, $M\in (K\cap \J_\rho)^{(\emptyset)}$ we 
define
\begin{equation} 
\label{e10}
\rho^K_M=\sum_{J\in \J_\rho: J\cap K=M}
\!\!\!\rho_J \hbox{ if } M\neq K \hbox{ and }
\rho^K_K=\!\! \sum_{J\in \J_\rho: 
J\cap K=K\vee J\cap K=\emptyset} \!\!\!\! \rho_M. 
\end{equation} 
$\Box$
\end{definition}
By definition the above quantities are positive: 
$\rho^K_M>0$ and $\rho^K_K>0$.
For all $K\in \Y(\J_\rho)$ we have
\begin{equation}
\label{e11}
\sum_{M\in (K\cap \J_\rho)^{(\emptyset)}} \rho^K_{M}=
\sum_{M\in K\cap \J_\rho}  
\; \sum_{J\in \J_\rho: J\cap K=M}\!\!\!\!\rho_J=
\sum_{J\in \J_\rho}\rho_J=1.
\end{equation}
Then, $\rho^K_\bullet=(\rho^K_{M}: M\in (K\cap \J_\rho)^{(\emptyset)})$
is a probability vector.

\medskip

In particular $\rho^J_J\ge \rho_J+\rho_{J^c}>0$ 
for all $J\in \J_\rho^{(I)}$ and when 
$I\in \J_\rho$ then,
$\rho^I_I=\rho_I$ and 
$\rho^I_K=\rho_K \hbox{ for } K\in \J_\rho^{(I)}$.

\begin{lemma}
\label{lemma2}
$(i)$ We have:
\begin{equation}
\label{e12}
\forall K\in \Y(\J_\rho), 
M\in (K\cap \J_\rho)^{(\emptyset)}, M\neq K: 
\;\; \rho^K_K<\rho^M_M.
\end{equation}
When we assume $\rho_J=\rho_{J^c}$ for $J\in \J_\rho^{(I)}$, we get
\begin{equation}
\label{e12'}
\forall K\in \Y(\J_\rho),
M\in (J\cap \J_\rho)^{(\emptyset)}, M\neq K:
\;\; \rho^K_M=\rho^K_{K\setminus M}.
\end{equation}

\noindent $(ii)$ The atoms of $\D^\rho$ are characterized by 
the following relation:
\begin{equation}
\label{e13}
\forall L\in \Y(\J_\rho): \;\; 
L\in \D^\rho \Leftrightarrow \rho^L_L=1.
\end{equation}
\end{lemma}

\begin{proof}
$(i)$ Let us show (\ref{e12}). For $J\in \J_\rho$ we have:
$$
\big[ J\cap K=M\Rightarrow J\cap M=M \big] \hbox{ and }
\big[ J\cap K=\emptyset\Rightarrow J\cap M=\emptyset\big]. 
$$
Hence, from definition (\ref{e10}) we get $\rho^K_K\le \rho^M_M$.
For showing the strict inequality we use that there exists some
$J\in \J_\rho^{(I)}$ such that $J\cap K=M$.
Note that $J^c\cap K\neq \emptyset$ but $J^c\cap M=\emptyset$, 
then $\rho^K_K<\rho^M_M$, so (\ref{e12}) is proven.

\medskip

The relation (\ref{e12'}) follows from
$$
J\in \J_\rho: \;\; J\cap K=M \Rightarrow J^c\cap K=K\setminus M,
$$
and $\rho_J=\rho_{J^c}$ for $J\in \J_\rho^{(I)}$.
In fact, both relations imply
$$
\sum_{J\in \J_\rho^{(I)}: J\cap K=M} \rho_J=
\sum_{J^c\in \J_\rho^{(I)}: J^c\cap K=K\setminus M} \rho_{J^c} \;.
$$

$(ii)$ Let us show the equivalence (\ref{e13}). The implication
($\Rightarrow$) is a direct consequence of 
$L\cap J=L$ or $L\cap J=\emptyset$ for $J\in  \J_\rho$.
The converse relation ($\Leftarrow$) is deduced from the fact
that $\rho^L_L=1$ happens if and only if  
$L\cap J=L$ or $L\cap J=\emptyset$ 
for $J\in  \J_\rho$, but since $L\in \Y(\J_\rho)$, 
from (\ref{e8}) we get that $L$ is necessarily 
an atom of $\D^\rho$.  
\end{proof}

\medskip

\begin{lemma}
\label{lemma3}
Let $K\in \Y(\J_\rho)$. Then, the marginal $\Xi[\mu]_K$ of 
$\Xi[\mu]$ on $K$, satisfies
$$
\Xi[\mu]_K=\rho^K_K \, \mu+ \sum_{M\in (K\cap \J_\rho)^{(\emptyset,K)}}
\!\! \rho^K_M \; \mu_M\otimes \mu_{K\setminus M}.
$$
\end{lemma}

\begin{proof}
Since $J\cap K=M$ implies $J^c\cap K=K\setminus M$, 
from (\ref{eab}) we obtain,
\begin{eqnarray*}
\Xi[\mu]_K&=&\sum_{J\in \AS^{\emptyset}}
\rho_J (\mu_J\otimes \mu_{J^c})_K =
\sum_{J\in \AS^{\emptyset}}
\rho_J \, \mu_{J\cap K}\otimes \mu_{J^c\cap K}\\
&=&
\left(\sum_{J\in \AS^{\emptyset}: J\cap K=K\vee J\cap K=\emptyset}                
\!\!\!\!\!\!\! \rho_J\right) \! \mu_K+
\sum_{M\in (K\cap \J_\rho)^{(\emptyset,K)}} \!
\left(\sum_{J\in \AS^{\emptyset}: J\cap K=M} \!\!\rho_J\right)
\mu_M\otimes \mu_{K\setminus M}\\
&=& \rho^K_K \,\mu_K+
\sum_{M\in (K\cap \J_\rho)^{(\emptyset,K)}}
\!\!\!\rho^K_M \, \mu_M\otimes \mu_{K\setminus M}.
\end{eqnarray*}
\end{proof}

\medskip

Let us define the following kernel $f^K_\D$
between sets $K\in \Y(\J_\rho)$
and dyadic partitions $\D\in \DD_{1,2}(K)$. We set 
\begin{equation}
\label{e15}
f^K(\{K\})=\rho^K_K \hbox{ and }
f^K(\{M, K\setminus M\})=\rho^K_M+\rho^K_{K\setminus M}
\hbox{ if } M\in (K\cap \J_\rho)^{(\emptyset, K)}.
\end{equation}
Then, equality (\ref{e11}) can be written 
in terms of dyadic partitions,
\begin{eqnarray}
\nonumber
\sum_{\D\in \DD_{1,2}(K)}\!\!\!\! f^K_\D&=&f^K(\{K\})+
\sum_{\{M, K\setminus M\}\in \DD_2(K)}\!\!\!\! 
f^K(\{M, K\!\setminus \!M\})\\
\label{e16}
&=&\rho^K_K+
\!\! \sum_{\{M, K\setminus M\}\in \DD_2(K)}
\!\!\!(\rho^K_M+\rho^K_{K\setminus M})=1.
\end{eqnarray}

In the following sections we will present our main results. 
We note that there will be  
cases in which these results will be trivial, 
for instance when $\J_\rho=\{I\}$ or $\J_\rho=\{J,J^c\}$ for 
some $J\in \AS^{(\emptyset,I)}$, 
but they will be not listed in detail.
We will assume that the sets $I$, $A_i$, $i\in I$, 
and $\J_\rho$, are sufficiently big in order 
that the statements of our results make sense 
and are not trivial.
 
\section{The recursive equation in terms of trees}
\label{sec3}

In this Section we supply the first of our main results, the
decomposition of $\Xi^n[\mu]$ in terms of product marginal 
measures, where the marginals are the atoms of some partitions.
This is done in Theorem \ref{theo0}. It requires to introduce some
dyadic trees because the atoms of the partitions are 
exactly the set of leaves of some dyadic trees.

\medskip

To expand $\Xi^n[\mu]$ for $n\ge 1$, we require 
to introduce further notation. Let us describe a class of
rooted dyadic trees whose nodes are sets, in fact they are 
elements of $\Y(\J_\rho)$. The dyadic property
means that each parent node has one or two children:
if it has one children the set associated to the children is 
the same as the one of the parent, and when it has two children 
the set of the parent is partitioned into two disjoint nonempty sets
by some set in $\J_\rho^{(I)}$, 
and these are the sets associated to the children. 
The set $I$ will be the root of all these trees. 

\medskip

Let us be more precise in notation and concepts.
We note by $\tT=\tT(\J_\rho)$ the family of dyadic 
trees rooted by $I$,
which depends on $\J_\rho$, and that we will construct in an inductive way.
The recursion will depend on the length $|T|$ of a tree $T\in \tT$, 
so the classes 
$\tT_n=\{T\in \tT: \, |T|=n\}$ will be defined
in a recursive way for $n\ge 0$.

\medskip

A tree $T\in \tT$ is defined as the set of its branches. 
A branch $b\in T$ is a tuple of elements in $\Y(\J_\rho)$ and 
its last component is called its leaf and noted $\ell(b)$.  
The set of leaves of the tree $T$ is 
$$
\partial(T)=\{\ell(b): b\in T\},
$$
and a leaf of $T$ is simply noted $\ell\in \partial(T)$. 
As a consequence of our construction of $\tT$, 
all the branches $b$ of a tree $T\in \tT$ will have
the same length, so $|b|=|T|$.
Any of these branches is written $b=(b_0,..,b_{|T|})$,
and so $b_{|T|}=\ell(b)$. 

\medskip

Below, the algorithm of construction of $\tT$ is given 
as a recursive definition of $(\tT_n: n\ge 0)$.

\medskip

For $n=0$, the class $\tT_0$ is a singleton formed by the unique 
tree $T=\{I\}$. So, it that has a unique branch  
$b=(I)$ with leaf $\ell(b)=I$. The length of $T$ is 
by definition $|T|=0$, so $|b|=0$ and $b_0=I$. 

\medskip

Assume we have constructed the set of trees $\tT_n$. We will
construct $\tT_{n+1}$ by using the following algorithm:

\medskip

Take $T\in \tT_n$. It generates a family of trees 
in $\tT_{n+1}$, where each one of these trees is the result 
of adding either one or two nodes, to each leaf 
of $T$. So, any of the choices made for the leaves
$\ell \in \partial T$, defines a tree $T'\in \tT_{n+1}$. 
To be precise let $b=(b_0,..,b_n)$ be a branch in $T$, then:

\begin{itemize}

\item If we are not in the case $(n=0,\rho_I=0)$, $b$
can generate the branch $b'=(b_0,..,b_n,b_n)$, 
so with $b_{n+1}=b_n$;

\item If $\ell(b)$ is not an atom in $\D^\rho$, $b$ 
can generate two branches $b'$ and $b''$. This is done by partitioning
the set $\ell(b)$ into a pair of nonempty sets
$\{\ell(b)\cap J, \ell(b)\cap J^c\}$
with some $J\in \J_\rho^{(I)}$. The two branches generated by $b$ are
respectively $b'=(b, \ell(b)\cap J)$ and $b''=(b, \ell(b)\cap J^c)$, so
these branches share all the nodes with $b$ except that we have added
to them and extra node, these are their leaves $\ell(b)\cap J$ and 
$\ell(b)\cap J^c$ respectively.

\end{itemize}

We have specified the possible choices on the branches of $T$, but 
as it can be easily checked
we could also performed it in terms of $\partial(T)$. 
As said, once a choice is made for the whole 
set of branches $\{b\in T\}$, or equivalently for all the leaves 
$\{\ell\in \partial(T)\}$, a tree $T'$ of length 
$|T'|=|T|+1$ is defined from $T$, or equivalently a partition 
$\partial(T')$ is defined from $\partial(T)$. When this happens we put 
\begin{equation}
\label{e17}
T\rightarrow T'  \hbox{ or equivalently } \partial(T)\rightarrow 
\partial(T'),
\end{equation}
which defines a relation in $\tT$ or equivalently in 
$\partial(\tT)=\{\partial(T): T\in \tT\}$. 

\medskip

Thus, a tree $T\in \tT_1$ can have the following shapes: 
either it has one branch $(I,I)$ in which case $I$ 
is the unique leaf (this can happen only when $\rho_I>0$); 
or it can have two branches $\{(I,J), (I,J^c)\}$
for some $J\in \J_\rho^{(I)}$, and so 
with leaves $J$ and $J^c$ respectively.

\medskip

The family of rooted trees constructed as above but with root $K$ 
instead of $I$, is noted by $\tT^K$. So,
$\tT_n^K$ refers to the class of the trees in $\tT^K$
of length $n$. With this notation
the recursive step to construct $\tT_{n+1}$
from $\tT_{n}$, can be summarized by saying that a tree
$T\in \tT_n$ generates a family of trees $T'\in \tT_{n+1}$,
each $T'$ is the result of a choice  
of a family of trees $({T'}^\ell\in \tT_1^\ell: \ell\in \partial(T))$, 
being ${T'}^\ell$ attached to $\ell$.

\medskip

In the next result, $\mu_\ell$ refers to the
marginal probability measure $\mu$ on the set $\ell$.

\begin{lemma}
\label{lemmafd}
For every $\mu\in {\cal P}_I$,
for all $n\ge 1$ and all $j\le n$, the following relation is
satisfied,
\begin{eqnarray}
\label{e18}
\Xi^n[\mu]&=&\sum_{T\in \tT_j} \left(\sum_{b\in T}
\prod_{r=1}^{|b|}\rho^{b_{r-1}}_{b_r}\right)
\bigotimes_{\ell\in \partial(T)} (\Xi^{n-j}[\mu])_\ell\\
\label{e22}
&=&\sum_{T\in \tT_j} \left(\sum_{b\in T}
\prod_{r=1}^{|b|}\rho^{b_{r-1}}_{b_r}\right)
\left(\bigotimes_{\ell\in \partial(T)\setminus \D^\rho}
\!\!\! \Xi^{n-j}[\mu]_\ell\right)\! \otimes \!
\left(\bigotimes_{\ell\in \partial(T)\cap D^\rho} \!\!\!\! \mu_\ell\right).
\end{eqnarray}
(We note $\prod_{r=1}^{|b|}=1$ when $|b|=0$.)
\end{lemma}

\begin{proof}
First of all, by using Lemma \ref{lemma1} $(i)$ the expression
(\ref{e18}) becomes (\ref{e22}). 
So, we only need to show (\ref{e18}).
This is done by recurrence on $n\ge 1$.

\medskip

Let $n=1$. The development made for the family of trees 
$\tT_1$, implies that the relation (\ref{e4}) can be written as
$$
\Xi[\mu]=\sum_{T\in \tT_1}
\left(\sum_{b\in T} \rho^{b_0}_{b_1}\right)
\bigotimes_{\ell\in \partial(T)}\mu_\ell.
$$
Then, (\ref{e18}) is satisfied for $n=1$.
Now, assume we have
shown (\ref{e18}) for some $n-1$, let show it for $n$. 

\medskip

First take $j<n$. By recurrence hypothesis, we can apply
formula (\ref{e18}) to $n-1$, $j$ and $\Xi[\mu]$. Hence, 
\begin{eqnarray*}
\Xi^n[\mu]&=&\Xi^{n-1}(\Xi[\mu])\\
&=&\sum_{T\in \tT_j} \left(\sum_{b\in T}
\prod_{r=1}^{|b|}\rho^{b_{r-1}}_{b_r}\right)
\bigotimes_{\ell\in \partial(T)} (\Xi^{n-1-j}[\Xi[\mu]])_{\ell}\\
&=&\sum_{T\in \tT_j} \left(\sum_{b\in T}
\prod_{r=1}^{|b|}\rho^{b_{r-1}}_{b_r}\right)
\bigotimes_{\ell\in \partial(T)} (\Xi^{n-j}[\mu])_\ell.
\end{eqnarray*}
Then the formula (\ref{e18}) holds for $n$, $j$ and $\mu$.

\medskip

Now take $j=n$. By recurrence hypothesis and by using 
Lemma \ref{lemma3} we get
\begin{eqnarray}
\nonumber
\Xi^n[\mu]&=&\Xi^{n-1}[\Xi[\mu]]\\
\nonumber
&=& \sum_{T\in \tT_{n-1}} \left(\sum_{b\in T}
\prod_{r=1}^{|b|}\rho^{b_{r-1}}_{b_r}\right)
\bigotimes_{\ell\in \partial(T)} \Xi[\mu]_\ell\\
&=&
\label{e21}
\sum_{T\in \tT_{n-1}} \!\!\! \left(\sum_{b\in T}
\prod_{r=1}^{|b|}\rho^{b_{r-1}}_{b_r}\!\right)\!
\bigotimes_{\ell\in \partial(T)} \!
\left(\sum_{{T'}^\ell\in \tT_1^\ell}\!\!
\left(\sum_{b'\in {T'}^\ell} \!\!\!\rho^{b'_0}_{b'_1}\right)\!\!
\bigotimes_{\ell'\in \partial({T'}^\ell)}
\!\!\!\!\mu_{\ell'}\right)\\
\label{e21x}
&=& 
\sum_{T^*\in \tT_n} \left(\sum_{b\in T^*}
\prod_{r=1}^{|b^*|}\rho^{b^*_{r-1}}_{b^*_r}\right)
\bigotimes_{\ell^*\in \partial(T^*)} \!\!\!\! \mu_{\ell^*}. 
\end{eqnarray}
In (\ref{e21}) we set $b_{|b|}=b'_0$. On the other hand, in 
(\ref{e21x}) we used, 
$$
\bigotimes_{\ell^*\in \partial(T^*)} \!\!\mu_{\ell^*}=
\bigotimes_{\ell\in \partial(T)}\left(\bigotimes_{\ell'\in \partial({T'}^\ell)}
\!\!\mu_{\ell'}\right),
$$
and
$$
\prod_{r=1}^{|b^*|}\rho^{b^*_{r-1}}_{b^*_r}=
\left(\prod_{r=1}^{|b|}\rho^{b_{r-1}}_{b_r}\right) \rho^{b'_0}_{b'_1}
\hbox{ for }
b^*_s=b_s \hbox{ for } s\le |b|, \hbox{ and } b^*_{|b|+1}=b'_1,
$$
for the tree $T^*$ formed by adding ${T'}^\ell\in \tT^\ell$ to each leaf 
$\ell\in \partial(T)$. 
Hence, the result is shown.
\end{proof}

\medskip

We will note by $\partial(\tT_n)=\{\partial(T): T\in \tT_n\}$.

\begin{theorem}
\label{theo0}
For every probability measure $\mu\in {\cal P}_I$ and 
all $n\ge 1$, we get the following decomposition
\begin{equation}
\label{eqwr1}
\Xi^n[\mu]=\sum_{\delta\in \partial(\tT_n)} q^n_\delta \;
\bigotimes_{\ell\in \delta} \mu_\ell,
\end{equation}
where the vector $q^n=(q^n_\delta: \delta\in \partial(\tT_n)$ is given by,
\begin{equation}
\label{eqwr2}
q^n_\delta=\sum_{T\in \tT_n: \partial(T)=\delta}
\left(\sum_{b\in T} \prod_{r=1}^{|b|}\rho^{b_{r-1}}_{b_r}\right),
\end{equation}
and it is a probability vector, so it satisfies
\begin{equation}
\label{eqwr3}
\sum_{\delta\in \partial(\tT_n)} q^n_\delta=1.
\end{equation}
\end{theorem}

\begin{proof}
By taking $j=n$ in (\ref{e18}) we get, 
\begin{equation}
\label{e19}
\Xi^n[\mu]=\sum_{T\in \tT_n} \left(\sum_{b\in T}
\prod_{r=1}^{|b|}\rho^{b_{r-1}}_{b_r}\right)
\bigotimes_{\ell\in \partial(T)} \mu_\ell.
\end{equation}
So, by using definition (\ref{eqwr2}), the equality (\ref{eqwr1}) 
is shown. Since (\ref{e19}) expresses that the probability measure 
$\Xi^n[\mu]$
is a positive linear combination of the set of probability measures
$(\bigotimes_{\ell\in \partial(T)} \mu_\ell: T\in \tT_n)$,
we deduce it is necessarily a convex linear combination, that is
$$
\sum_{T\in \tT_n}\left(\sum_{b\in T}
\prod_{r=1}^{|b|}\rho^{b_{r-1}}_{b_r}\right)=1.
$$
But this is exactly (\ref{eqwr3}). The result is shown.
\end{proof}

\smallskip

Then, in the expansion (\ref{eqwr1}) $\Xi^n[\mu]$ 
has a weight $q^n_\delta$
of being the product probability measure
$\otimes_{\ell\in \delta} \mu_\ell$.

\begin{remark}
\label{remirs}
In the following section we will use some properties of 
the relation $\rightarrow$ 
on $\partial(\tT)$ defined in (\ref{e17}). We have
that $\rightarrow$ is an order relation and 
$\delta\rightarrow \delta'$ implies that $\delta'$ is finer
than $\delta$ (finer includes equal). Also,
for all $\delta\in \partial(\tT)$, $\delta\neq \{I\}$, 
there exists a path $\delta_1=\{I\}\rightarrow...\rightarrow
\delta_k=\delta$ from $\{I\}$ to $\delta$, and $\{I\}\rightarrow 
\{I\}$ only when $\rho_I>0$. 
On the ordered space $(\partial(\tT), \rightarrow)$ 
we can say that $\delta'$ is a 
successor of $\delta$ when $\delta\to \delta'$ in a consistent 
way because $(\delta_1\rightarrow...\rightarrow
\delta_k, \, \delta_1\neq \delta_k)$ implies 
$\delta_k\not\rightarrow \delta_1$. But the ordered space 
$(\partial(\tT),\rightarrow)$ is in general not a tree. 
For instance if the elements $I, J_1, J_2$ are three
different elements of $\J_\rho$, and the intersections 
$J_1\cap J_2$, $J_1^c\cap J_2$, $J_1^c\cap J_2$ and  
$J_1^c\cap J_2^c$ are nonempty,
then
$$
\{I\}\to \{J_1, J_1^c\}\to \{J_1\cap J_2, J_1^c\cap J_2, 
J_2^c\}\to \{J_1\cap J_2, J_1^c\cap J_2, J_1\cap J_2^c, J_1^c\cap J_2^c\}
$$
and  
$$
\{I\}\to \{J_2, J_2^c\}\to \{J_1\cap J_2, J_1\cap J_2^c,
J_1^c\}\to \{J_1\cap J_2, J_1\cap J_2^c, J_1^c\cap J_2, J_1^c\cap 
J_2^c\} 
$$
are two different paths from $\{I\}$ to 
$\{J_1\cap J_2, J_1^c\cap J_2, J_1\cap J_2^c, J_1^c\cap J_2^c\}$,
having in common only the initial and final points. So, 
$\{J_1\cap J_2, J_1^c\cap J_2, J_1\cap J_2^c, J_1^c\cap J_2^c\}$
has at least two predecessors. 
\end{remark}

\section{Markov chain, geometric convergence 
and quasi-stationarity} 
\label{sec4}

In this section we supply our main results. Firstly, the 
definition of a natural Markov chain associated to 
$(\Xi^n: n\ge 0)$ is done
in Lemmas \ref{lemma4} and \ref{lemma5}.
The main results are the description of this chain
found in Theorem \ref{theo1}.
Since the orbit $(\Xi^n[\mu])$,
converges to the product of the marginal probability
measures on the atoms of the partition $\D^\rho$, 
we study geometric convergence to the limit probability measure. 
We give the geometric decay rate,   
and we study the ratio limit and the 
quasi-stationary behavior of the chain. 
This last study responds to the following question: if 
the chain has not arrived 
to the limit probability 
measure after a long time, which is its distribution?
Finally, in Corollary \ref{cor1} we supply the Markov chain 
that never hit the limit distribution.

\medskip

The relations 
(\ref{eqwr1}), (\ref{eqwr2}) and (\ref{eqwr3}) of Theorem 
\ref{theo0} will be at the basis of the construction
of a Markov chain $Y=(Y_n: n\ge 0)$ taking values on 
$\partial(\tT)$ and having 
the following remarkable property: if it starts from $Y_0=\{I\}$, then   
at time $n$, the event $\{Y_n=\delta\}$ has probability $q^n_\delta$.

\medskip

In this purpose we define the following transition matrix
$P=(P_{\delta, \delta'}: \delta, \delta'\in \partial(\tT))$. 
First we put $P_{\delta, \delta'}=0$ when 
$\delta\not\rightarrow \delta'$.

\medskip

To define the transition probability $P_{\delta, \delta'}$ when
$\delta\rightarrow \delta'$ it is useful to introduce the following 
notation: for each leaf $\ell\in \delta$ we denote by
$\{\ell_1,\ell_2\}$ its corresponding dyadic 
partition in $\delta'$. We can either have
$\{\ell_1,\ell_2\}=\{\ell\}$ that is $\ell_1=\ell_2=\ell$ 
which means $\ell\in \delta\cap \delta'$;
or $\{\ell_1,\ell_2\}\in \DD_2(\ell)$
is an strictly dyadic partition of $\ell$ and in this case 
$\ell\in \delta\setminus \delta'$. We define
\begin{equation}
\label{e23}
\forall \delta, \delta'\!\in \!\partial(\tT), \delta\rightarrow 
\delta':\;
P_{\delta, \delta'}\!=\!
\prod_{\ell\in \delta} f^\ell(\{\ell_1, \ell_2\})
\!=\!\left(\prod_{\ell\in \delta\cap \delta'}\! 
\rho^\ell_\ell\right) 
\left(\prod_{\ell\in \delta\setminus \delta'}\!\!
(\rho^\ell_{\ell_1}\!+\!\rho^\ell_{\ell_2})\right).
\end{equation}
In particular
\begin{equation}
\label{e24}
\forall \delta\in \partial(\tT):\;\,
P_{\delta, \delta}=\prod_{\ell\in \delta} \rho^\ell_\ell.
\end{equation}

\begin{lemma}
\label{lemma4}
$P$ is an stochastic transition matrix, that is
$$
\forall \delta\in \partial(\tT): \quad 
\sum_{\delta'\in \partial(\tT): \delta\to \delta'}
\!\!\!P_{\delta, \delta'}=1.
$$
\end{lemma}

\begin{proof}
We will use the following decomposition: $\delta=(\delta\cap \D^\rho) \cup 
(\delta\setminus \D^\rho)$, so the atoms of $\delta$ are partitioned 
according to the fact that if they belong or not to $\D^\rho$.
We recall that $\DD_2(\ell)$  
excludes the partition $\{\ell\}$.
For $U\subseteq \delta\setminus \D^\rho$ denote
$$
{\cal D}(U,2)=\{((K_1^\ell, K_2^\ell): \ell\in U)\in \prod_{\ell\in U}
\DD_2(\ell): \forall \ell\in U, \exists J\in \J_\rho,  
K_1^\ell=\ell\cap J, K_2^\ell=\ell\cap J^c\}.
$$
We have
\begin{eqnarray*}
\sum_{\delta': \delta\rightarrow \delta'} \!\! P_{\delta, \delta'}&=&
\!\!\left(\prod_{\ell\in \delta\cap \D^\rho} \!\!\! \rho^\ell_\ell\right)\!
\times \! \left( \sum_{U\subseteq \delta\setminus \D^\rho}
\!\!\left(\prod_{\ell\in U^c} \rho^\ell_\ell\right)
\left(\sum_{((K_1^\ell, K_2^\ell): \ell\in U)\in {\cal D}(U,2)} 
\, \prod_{\ell\in U}
(\rho^\ell_{K_1^\ell}\!+\!\rho^\ell_{K_2^\ell})
\right) \! \right)\\
&=&\!\sum_{U\subseteq \delta\setminus \D^\rho} \!\!
\left(\prod_{\ell\in U} \rho^\ell_\ell\right)
\left(\sum_{((K_1^\ell, K_2^\ell): \ell\in U)\in {\cal D}(U,2)} \;
\prod_{\ell\in U}(\rho^\ell_{K_1^\ell}+ \rho^\ell_{K_2^\ell})\right).
\end{eqnarray*}
This last equality uses $\rho^\ell_\ell=1$ when $\ell\in 
\D^\rho$, see (\ref{e13}) in Lemma \ref{lemma2} $(ii)$. 
By using notation $f^\ell(\gamma_\ell)$ introduced in (\ref{e15})
we have,
\begin{eqnarray*}
\sum_{\delta': \delta\rightarrow \delta'} \!P_{\delta, \delta'}&=&
\sum_{(\gamma^\ell: \ell\in \delta\setminus\D^\rho)\in 
\prod\limits_{\ell\in \delta\setminus\D^\rho}\DD_{1,2}(\ell)}
\;\, \prod_{\ell\in \delta\setminus\D^\rho}f^\ell(\gamma_\ell)\\
&=& \prod\limits_{\ell\in \delta\setminus\D^\rho}
\left(\sum_{\gamma_\ell\in \DD_{1,2}(\ell)} 
f^\ell(\gamma_\ell)\right)=1. 
\end{eqnarray*}
In this last equality we use (\ref{e16}).
\end{proof}

\begin{remark}
\label{remirr}
From the positive properties of coefficients $\rho^K_M$ we get that
$P_{\delta, \delta'}>0$ if and only if $\delta\rightarrow \delta'$.
Since there exists a path $\delta_1=\{I\}\rightarrow...\rightarrow 
\delta_k=\delta$ for all $\delta\in \partial(\tT)$, $\delta\neq 
\{I\}$, this path has positive probability.
\end{remark}

Let $Y=(Y_n: n\ge 0)$ be the Markov chain taking values in 
$\partial(\tT)$ defined by the transition stochastic 
matrix $P$. Let $(\Omega, {\cal F})$ be the measurable space with  
$\Omega=\partial(\tT)^\NN$ and $\FF$ the product $\sigma-$field. 
Let $(\PP_\delta: \delta\in \partial(\tT))$ be the family of
probability Markov measures on $(\Omega, {\cal F})$, all of them
with transition matrix $P$, and $\PP_\delta$  
starting from $\delta$. We will simply note $\PP:=\PP_{\{I\}}$, 
because most of the time the chain will assume to start from 
$Y_0=\{I\}$,
and this will be clear from the context or the notation. 
The mean expected values associated to $\PP_\delta$ and $\PP$
are noted by $\EE_\delta$ and $\EE$, respectively.

\medskip

The Markov chain $(Y_n: n\ge 0)$ can be also constructed
from a probability space 
$({\widetilde{\Omega}}, {\widetilde{\cal F}},{\bf P})$ 
containing an independent family of random variables 
$\left(\delta^K_n: K\in \Y(\J_\rho), n\ge 1 \right)$,  
where $\delta^K_n$ takes values in $\DD_{1,2}(K)$ and 
$$
{\bf P}(\delta^K_n=\delta)=f^K_\delta,
$$
where $f^K_\delta$ was defined in (\ref{e15}).
Thus, the random variables $(\delta^K_n: n\ge 1)$ are independent 
and identically distributed with law $f^K_\bullet$. It is easily
checked that the random sequence given by
$$
Y_0=\delta, \;\; Y_n=(\delta^K_n: K\in Y_{n-1})  \;\, \forall n\ge 1,
$$
defines a Markov chain $(Y_n)$ starting from $\delta$, and transition 
probability given by (\ref{e23}).

\medskip

Let us show that the Markov chain $(Y_n)$ fulfills the first claim
of this section: after $n-$steps of time  
the probability of the event $\{Y_n=\delta\}$ 
is the weight of all 
the trees of length $j$ whose set of leaves is $\delta$.

\begin{lemma}
\label{lemma5}
For every $n\ge0$ and $\delta\in \partial(\tT)$ it holds 
$\PP(Y_n=\delta)=q^n_\delta$.
\end{lemma}

\begin{proof}
We will use a recurrence argument.
For $n=0$ the property holds because $Y_0=\{I\}$ and 
the class of trees of length $0$ is the singleton $\tT_0=\{\{I\}\}$. 
Assume the property holds up to $n$ let us show it for $n+1$. We have
\begin{eqnarray*}
\PP(Y_{n+1}=\delta')&=& 
\sum_{\delta\in \partial(\tT): \delta \rightarrow \delta'}
\PP(Y_n=\delta) P_{\delta,\delta'}\\
&=&\sum_{\delta\in \partial(\tT): \delta \rightarrow \delta'}
q^n_\delta \, \left(\prod_{\ell\in \delta\cap \delta'}
\rho^\ell_\ell\right)
\left(\prod_{\ell\in \delta\setminus \delta'}\!\!
(\rho^\ell_{\ell_1}\!+\!\rho^\ell_{\ell_2})\right).
\end{eqnarray*}
Now we use the step (\ref{e21}) of the proof of Lemma \ref{lemmafd},
which allows to get $\PP(Y_{n+1}=\delta')=q^{n+1}_{\delta'}$. 
The result is proven.
\end{proof}

The partition $\D^\rho$ is an absorbing state for the chain $(Y_n)$ 
because $P_{\D^\rho,\D^\rho}=\prod_{L\in \D^\rho}\rho^L_L=1$, 
and so $Y_n=\D^\rho$ implies $Y_{n+k}=\D^\rho$ for all $k\ge 0$.

\medskip

Let us define the hitting times,
$$
\forall B\subseteq \partial(\tT): \quad \zeta_B=\inf\{n\ge 0: Y_n\in B\}.
$$
For singletons we simply put, 
$$
\forall \delta\in \partial(\tT): \quad \zeta_\delta= \zeta_{\{\delta\}}.
$$
For $\delta=\{I\}$ we have $\PP(\zeta_{\{I\}}=0)=1$. The random time
for attaining $\D^\rho$,
$$
\zeta=\zeta_{\D^\rho}=\inf\{n\ge 0: Y_n=\D^\rho\},
$$
is an absorbing time because $Y_{\zeta+n}=\D^\rho$ for all 
$n\ge 0$.

\medskip

Since $Y_n(\omega)\in \partial(\tT)$ we can define the random probability:
$$
\forall \omega\in \Omega: 
\quad \Xi^n[\mu](\omega)=\bigotimes_{K\in Y_n(\omega)} \mu_K.
$$
 From above discussion and Lemma \ref{lemma1} $(ii)$ 
we find,
$$
\forall n\ge 0: \;\; 
\Xi^{\zeta(\omega)+n}[\mu](\omega)=\bigotimes_{L\in \D^\rho} \mu_L.
$$

\begin{remark}
\label{eqtodas}
Note that 
\begin{equation}
\label{eq27}
\{\Xi^n[\mu]\neq \otimes_{L\in \D^\rho}\}\subseteq \{\zeta\!>\!n\}
\hbox{ and so }
\PP\left(\Xi^n[\mu]\neq \otimes_{L\in \D^\rho}
\mu_L \, \right)\le \PP(\zeta\!>\! n).
\end{equation}
It can be checked that when the spaces 
$I$, $A_i$, $i=1,..,n$, have sufficiently many points
we have the equivalence
$$
\big\{\forall \mu\in {\cal P}_I: \; 
\Xi^n[\mu]\neq \otimes_{L\in \D^\rho}\big\}=\{\zeta>n\}.
$$
For some particular 
${\widetilde{\mu}}\in {\cal P}_I$ the inequality (\ref{eq27}) 
can be strict. For instance, if 
${\widetilde{\mu}}=\otimes_{L\in \D^\rho} {\widetilde{\mu}}_L$ 
then
$\Xi^n[{\widetilde{\mu}}](\omega)={\widetilde{\mu}}$ for all $n\ge0$,
but $\PP(\zeta>0)=1$ in the nontrivial case $\D^\rho\neq \{I\}$.
\end{remark}

\medskip

In the next result we show that the random measure $\Xi^n[\mu](\omega)$ 
converges geometrically to a product measure with
the marginals of $\mu$ at the atoms of $\D^\rho$. This is controlled 
with the geometric decay rate of $\PP(\zeta> n)$. Also
we give the quasi-limiting distribution which results
from conditioning to the event $\{\zeta> n\}$ for $n\to \infty$.

\medskip

We will supply the notions of quasi-limiting distribution 
(and further of quasi-stationary distributions) 
in the context of the Markov chain $(Y_n)$. 
The definition and study of these concepts in the context of finite Markov 
chains which are irreducible on the non-absorbing states
are found in the pioneer work \cite{ds} and the 
continuous time case can be seen in Chapter $3$ of 
monograph \cite{cms}. There is a large body of literature
on quasi stationary distributions, 
in particular for extinction in population dynamics
and we recommend addressing to \cite{pp} for an exhaustive 
list of references.

\medskip

We emphasize
that $(Y_n)$ is not irreducible on $\partial(\tT)\setminus \{\D^\rho\}$, 
because when $(Y_n)$ exits from some state it does never return to it. 
In fact, $\delta_1\rightarrow
\delta_2\rightarrow ... \rightarrow \delta_k$ and
$\delta_k\neq \delta_1$ implies $\delta_k\not\rightarrow \delta_1$
(see Remark \ref{remirs}).
Therefore, we cannot apply
Perron-Frobenius theory which is in the theoretical basis of 
the main results of quasi-stationary distributions on finite Markov chains. 
So, we need to develop new elements to describe the
quasi-limiting behavior and in particular the geometric decay rate. 

\medskip

In this purpose we introduce a class of distinguished partitions in 
$\partial(\tT)$. Any $K\in \Y(\J_\rho)$ defines the partition
$$
\D^{\rho,K}=\{L\in \D^\rho: L\cap J=\emptyset\}\cup \{K\}. 
$$
So, the partition $\D^{\rho,K}$ has 
the same atoms as $\D^\rho$ when they do not
intersect $K$, and all the other atoms collapse into the unique atom 
$K\in \D^{\rho,K}$. 
For $a\in [0,1]$ define the following classes of sets and partitions,
\begin{equation}
\label{e50d}
\E(a)=\{K\in \Y(\J_\rho): \rho^K_K=a\}
\hbox{ and } \partial(\tT)^{\E(a)}=\{\D^{\rho,K}: K\in \E(a)\}.
\end{equation}
Note that $\E(a)$ and so $\partial(\tT)^{\E(a)}$ can be empty. When $a=1$
we have $\E(1)=\D^\rho$ and $\partial(\tT)^{\E(1)}=\{\D^\rho\}$.
When $\rho_I>0$ we have $\E(\rho_I)=\{I\}$ and 
$\partial(\tT)^{\E(\rho_I)}=\{\{I\}\}$.

\medskip

\begin{theorem}
\label{theo1}
Assume $\rho_I<1$. Then,
\begin{equation}
\label{e29}
\PP(\zeta<\infty)=1.
\end{equation}
Define
$$
\eta=\max\{\rho^K_K: K\in \Y(\J_\rho), K\not\in \D^\rho\}.
$$
Then $\eta\in (0,1)$. Let
\begin{equation}
\label{e31}
\E=\E(\eta),\quad \partial(\tT)^\E=\partial(\tT)^{\E(\eta)},\quad
\zeta^{\E}=\zeta_{\partial(\tT)^\E}.
\end{equation}
Then $0<\PP(\zeta^{\E}<\infty)<1$ and the geometric rate of decay of
$\PP(\zeta>n)$ satisfies,
\begin{equation}
\label{50e}
\lim\limits_{n\to \infty} \eta^{-n} \PP(\zeta\!>\!n)=
\lim\limits_{n\to \infty} \eta^{-n} 
\PP(\zeta\!>\!n, Y_n\!\in \!\partial(\tT)^\E)
=\EE\left(\eta^{-\zeta^\E}, \, \zeta^\E\!<\!\infty \right)
\!\in \! (0,\infty).
\end{equation}
The quasi-limiting distribution on  
$\partial(\tT)\setminus \{\D^\rho\}$ is given by,
\begin{eqnarray}
\nonumber
\forall \delta\in \partial(\tT)^\E:&{}&
\lim\limits_{n\to \infty} \PP(Y_n=\delta\,| \, \zeta>n)=
\frac{ \EE\left(\eta^{-\zeta_\delta}, \, \zeta_{\delta}<\infty \right)}
{\EE\left(\eta^{-\zeta^{\E}}, \, \zeta^{\E}<\infty\right)},\\
\label{e32}
\forall \delta\in \partial(\tT)\setminus \partial(\tT)^\E:&{}&
\lim\limits_{n\to \infty} \PP(Y_n=\delta\,| \, \zeta>n)=0.
\end{eqnarray}
Furthermore, we have the following ratio limit relation for 
$\delta\in \partial(\tT)\setminus \{\D^\rho\}$,
\begin{equation}
\label{50b}
\lim\limits_{n\to \infty}
\frac{\PP_\delta(\zeta>n)}{\PP(\zeta>n)}
=\frac{\EE_\delta(\eta^{-\zeta^\E}, \zeta^\E<\infty)}
{\EE(\eta^{-\zeta^\E}, \zeta^\E<\infty)}.
\end{equation}
Both ratios vanish only when $\PP_\delta(\zeta^\E<\infty)=0$.
Finally, the vector 
\begin{equation}
\label{rev1}
\varphi=(\varphi_\delta: \delta\in 
\partial(\tT)\setminus 
\{\D^\rho\}) \hbox{ with }
\varphi_\delta=\EE_\delta(\eta^{-\zeta^\E}, \zeta^\E<\infty),
\end{equation}
is a right eigenvector of the restriction of 
$P$ to $\partial(\tT)\setminus \{\D^\rho\}$, and it has 
eigenvalue $\eta$.
\end{theorem}

\noindent {\it Proof}.
It is obvious that $\eta>0$ and from (\ref{e13}) 
in Lemma \ref{lemma2} $(ii)$ we have $\eta<1$. 
Then $\E\cap \D^\rho=\emptyset$ because 
$K\in \E$ and $L\in \D^\rho$ imply
$\rho^K_K=\eta<1=\rho^L_L$. 
Note that if $\delta=D^{\rho,K}$ then
$P_{\delta,\delta}=\rho^K_K$. 
Hence 
\begin{equation}
\label{e33}
\forall \delta\in \partial(\tT)^\E: \quad
P_{\delta,\delta}=\eta.
\end{equation}
We claim that  
$$
\max\{P_{\delta,\delta}: \delta\in \partial(\tT),
\delta\neq \D^\rho\}= \eta.
$$
This follows from (\ref{e33}) for partitions having at most 
one atom that is not in $\D^\rho$, and if 
$\delta'\in \partial(\tT)$
has at least two different atoms $K,K'$ that are not elements
of $\D^\rho$, from (\ref{e24}) we get
$P_{\delta',\delta'}\le \eta^2$.

\medskip

Let
\begin{equation}
\label{eqlast}
\beta_0=\max\{P_{\delta,\delta}: \delta\in \partial(\tT),
\delta\neq \D^\rho,
\delta\!\not\in \!\partial(\tT)^\E\}.
\end{equation}
We have
\begin{equation}
\label{e33ax}
\beta_0\le \max\{\beta,\eta^2\}<\eta \, \hbox{ where }
\beta=\sup\{\rho^K_K: K\in \Y(\J_\rho), 
\rho^K_K< \eta\}.
\end{equation}
In fact, when $K\not\in \E\cup\{D^\rho\}$ we have
$P_{\D^{\rho,K}, \D^{\rho,K}}=\rho^K_K\le \beta$ and if
a partition has at least two different atoms $K,K'$ 
that are not in $\D^\rho$, then
$P_{\delta',\delta'}\le \eta^2$. Then, (\ref{e33ax}) is shown.

\medskip

Let us show (\ref{e29}). We use that when $(Y_n)$ exits 
from some state it does never 
return to it and inequality $P_{\delta,\delta}<1$ for 
$\delta\neq \D^\rho$. In fact, they allow us to prove 
that the Markov chain $(Y_n)$ visits every
state $\delta\neq \D^\rho$ only a finite number of times $\PP-$a.s.,
$$ 
\forall \delta\in \partial(\tT), \delta\neq \D^\rho: \quad 
\PP(\#\{n: Y_n=\delta\}<\infty)=1.
$$
Then, by using that $\D^\rho$ is an absorbing state, we obtain
(\ref{e29}),
$$
\PP(\exists n: Y_n=\D^\rho)=\PP(\zeta<\infty)=1.
$$

The existence of paths from $\{I\}$ to $\partial(\tT)^\E$ with 
positive probability gives $\PP(\zeta^\E<\infty)>0$.
On the other hand there exists 
$\delta'\in \partial(\tT)$
with $\delta'\rightarrow \D^\rho$ and 
$\#\{J\in \delta': J\not\in \D^\rho\}>1$. 
The existence of some path from $\{I\}$ to $\delta'$
with positive probability now gives
$\PP(\zeta^\E<\infty)<1$. We have shown
$\PP(\zeta^\E<\infty)\in (0,1)$.

\medskip

Let us now turn to the proof of relations (\ref{50e}), (\ref{e32}) 
and (\ref{50b}). We have
\begin{equation}
\label{eq41}
\forall \, \delta\in \partial(\tT)^\E, \, j\ge 0:\quad
\delta\rightarrow \delta' \Leftrightarrow \,
\big[\, \delta'=\delta \vee \delta'=\D^\rho\big].
\end{equation}
Then, the definition of $\E$ and $\partial(\tT)^\E$ in (\ref{e31})
and the fact that $\D^\rho$ is absorbing, allow us to get 
$$
\forall n\ge 0,\, \delta\in \partial(\tT)^\E:\quad  
\PP_\delta(Y_n=\delta)=\eta^n.
$$

We have
\begin{equation}
\label{e34}
\PP(\zeta>n)=\PP(\zeta>n, Y_n\not\in \partial(\tT)^\E)+
\PP(\zeta>n, Y_n\in \partial(\tT)^\E).
\end{equation}
Since $P_{\delta, \delta'}>0$ when $\delta\rightarrow \delta'$ 
and there exists paths of positive probability from $\{I\}$ 
to $\delta\in \partial(\tT)$, $\delta\neq \{I\}$ 
(see Remark \ref{remirr}), we obtain the existence
of $k_0\ge 1$ such that 
$$
\forall \, K\in \E:\quad \PP(\zeta_{\D^{\rho,K}} \le k_0)>0.
$$
So, 
$$
\alpha(\E):=
\min\{\PP(\zeta^{\D^{\rho,K}} \le k_0): K\in \E\}>0.
$$
Then, from the Markov property we get, 
\begin{eqnarray}
\nonumber
\PP(\zeta>n)&\ge& \sum_{j=1}^{k_0}\PP(\zeta^{\D^{\rho,K}}=j, \zeta>n)\\
\nonumber
&\ge & \sum_{j=1}^{k_0}\PP(\zeta^{\D^{\rho,K}}=j)
\PP_{D^{\rho,K}}(\zeta>n-j)\\
\nonumber
&\ge & 
\sum_{j=1}^{k_0}\PP(\zeta^{\D^{\rho,K}}=j)
\PP_{D^{\rho,K}}(Y_{n-j}=D^{\rho,K})\\
\nonumber
&\ge & \sum_{j=1}^{k_0}\PP(\zeta^{\D^{\rho,K}}=j) \eta^{n-j}\\
\label{e35}
&\ge & \alpha(\E) \eta^n.
\end{eqnarray}

To analyze the first term at the right hand side of 
equality (\ref{e34}) it will useful to first prove 
the following result, which uses the quantity $\beta_0$
defined (\ref{eqlast}) which satisfies $\beta_0<\eta<1$, 
see (\ref{e33ax}).

\begin{lemma}
\label{lemma6}
We have,
\begin{equation} 
\label{e36}
\forall\, \theta\!>\!0 \, \exists C'\!=\!C'(\theta):\quad 
\PP(\forall j\!\le \! n: \; Y_j\not\in (\partial(\tT)^\E \cup \{\D^\rho\})
\le C'(\beta_0\!+\!\theta)^n.
\end{equation} 
\end{lemma}

\noindent{\it Proof of Lemma \ref{lemma6}}.
Let $U=\partial(\tT)\setminus (\partial(\tT)^\E\cup \{\D^\rho\})$.
Put $\delta_1=\{I\}$. For every $s\ge 1$ denote by 
$$
{\cal C}(U,s)=\{(\delta_1,..,\delta_s)\in U^s:
\forall r\le s-1,\;  \delta_r\to \delta_{r+1}\hbox{ and } 
\delta_r\neq \delta_{r+1}\}.
$$
(So, $P_{\delta_r,\delta_{r+1}}>0$ for
all $r=1,..,s-1$, see Remark \ref{remirr}).
We have
\begin{eqnarray*}
&{}& \PP(\forall j\le n: \; Y_j\in U)\\
&{}& =\sum_{s\ge 1}\; \sum_{(\delta_1,..,\delta_s)\in {\cal C}(U,s)} \;
\prod_{r=1}^{s-1}
P_{\delta_r,\delta_{r+1}} \; \left(
\sum_{k_1,..,k_s\ge 0: \sum_{r=1}^s k_r=n-s} 
P^{k_r}_{\delta_r,\delta_r}\right). 
\end{eqnarray*}
When $(\delta_1,..,\delta_s)\in {\cal C}(U,s)$ we have 
that every $\delta_k$ with $k\le s$ satisfies 
$P_{\delta_k,\delta_k}\le \beta_0$. On the other hand,
$$
\#\{(k_1,..,k_s): \forall r\le s, \; k_r\ge 0; \; \sum_{r=1}^s 
k_r=n\!-\!s\}=\binom{n\!-\!1}{s}.
$$
Then,
$$
\PP(\forall j\le n: \; Y_j\in U)\le  
\sum_{s\ge 1} \binom{n\!-\!1}{s} \beta_0^{n-s} 
\left(\sum_{(\delta_1,..,\delta_s)\in {\cal C}(U,s)}
\; \prod_{r=0}^{s-1}P_{\delta_r,\delta_{r+1}} \right). 
$$

We claim that there exists a constant $k^*$ such that
${\cal C}(U,s)\neq \emptyset$ implies $s\le k^*$.
Let us show it. Fix an atom $L\in \D^\rho$. Let
$(K_n: n\ge 1)$ be a sequence of sets constructed
in an inductive way and satisfying the following properties:
$K_1=I$;  $L\subseteq K_n$ for all $n$;
$K_{n+1}=K_n\cap J_n$ for some $J_n\in \J_\rho$
and $K_{n+1}\subset K_n$ for all $n$.
Then, after a number $n_0$ of steps bounded by $\#I-\#L-1$ one
necessarily has $K_{n_0}=L$ and the construction is stopped.
Now, define
$k^*=\sum_{ L\in \D^\rho}(\#I-\#L-1)$. A consequence of 
the above argument is that the existence of some 
$(\delta_1,..,\delta_s)\in {\cal C}(U,s)$ implies
$s\le k^*$. So,
$$
C_1=\sum_{s\ge 1} \sum_{(\delta_1,..,\delta_s)\in {\cal C}(U,s)}
\; \prod_{r=0}^{s-1}P_{\delta_r,\delta_{r+1}}<\infty.
$$
On the other hand, for $\theta'\in (0,1)$ we have
$$
C_2(\theta')=\max_{s\le k^*}\; \sup_{n\ge 1} \binom{n-1}{s} 
(1-\theta')^{n-k^*}<\infty. 
$$
Then
$$
\PP(\forall j\le n: \; Y_j\in U)\le  C_1\cdot C_2(\theta') 
\beta_0^{n-k^*}/ (1-\theta')^{n-k^*}.
$$
So by taking $\theta'\in (0,1)$ such that
$\beta_0/ (1-\theta')<\beta_0+\theta$ we get that the constant
$$
C'=(\beta_0+\theta)^{-k^*} C_1\cdot C_2(\theta')
$$ 
makes the job in (\ref{e36}). 
$\Box$ 

\bigskip

\noindent {\it Continuation with the proof of Theorem \ref{theo1}}.

In (\ref{e36}) we will always take $\theta>0$ such that 
$\beta_0+\theta<\eta$. Hence, from (\ref{e35}) and 
(\ref{e36}) we find,
\begin{equation}
\label{50a}
\PP(Y_n\not\in \partial(\tT)^\E \, | \, \zeta>n)\le 
C'' \left((\beta_0+\theta)/\eta\right)^n\to 0
\hbox{ as } n\to \infty,
\end{equation}
with $C''=C'/\alpha(\E)$. Therefore,
\begin{equation}
\label{50f}
\lim\limits_{n\to \infty}\PP(Y_n\in \partial(\tT)^\E \, | \, \zeta>n)=1.
\end{equation}
Let us examine the second term at the right hand side 
of equality (\ref{e34}). For every $K\in \E$ we have
\begin{eqnarray*}
\PP(\zeta>n, Y_n=D^{\rho,K})&=&\sum_{j=1}^n 
\PP(\zeta>n, \zeta_{D^{\rho,K}}=j)\\
&=&\sum_{j=1}^n\PP(\zeta_{D^{\rho,K}}=j)
\PP_{D^{\rho,K}}(\zeta>n-j)\\
&=&\sum_{j=1}^n\PP(\zeta_{D^{\rho,K}}=j)\eta^{n-j}\\ 
&=&\eta^n
\left(\sum_{j=1}^n\eta^{-j}\PP(\zeta_{D^{\rho,K}}=j)\right).
\end{eqnarray*}
Since 
\begin{eqnarray*}
\PP(\zeta_{D^{\rho,K}}=j)&\le& \PP(\zeta^\E=j)\\
&\le&\PP(\forall n\le j-1: \; Y_n\not\in (\partial(\tT)^\E)\cup \{\D^\rho\})
\le C'(\beta_0+\theta)^{j-1},
\end{eqnarray*}
and $\beta_0+\epsilon<\eta$, we get 
$$
\sum_{j=1}^\infty \eta^{-j}\PP(\zeta_{D^{\rho,K}}=j)<\infty.
$$
Hence, 
\begin{eqnarray}
\label{e40y}
\forall K\in \E:\; 
\lim\limits_{n\to \infty}\eta^{-n}\PP(\zeta>n, Y_n=D^{\rho,K})&=&
\sum_{j=1}^\infty\eta^{-j}\PP(\zeta_{D^{\rho,K}}=j)\\
\nonumber
&=&
\EE\left(\eta^{-\zeta_{D^{\rho,K}}}, \zeta_{D^{\rho,K}}<\infty 
\right)<\infty.
\end{eqnarray}
We have
$$
\zeta_{D^{\rho,K}}<\infty \, \Rightarrow \, 
\big[\, \forall K'\in \E\setminus \{K\}: \; \zeta_{D^{\rho,K}}=\infty 
\hbox{ and } \zeta^\E=\zeta_{D^{\rho,K}} \, \big].
$$
Then,
$$
\{\zeta^\E=j\}=\bigcup_{K\in \E} \{\zeta_{D^{\rho,K}}=j\}
$$
and the union is disjoint. Hence,
$$
\eta^{-\zeta^\E} {\bf 1}_{\zeta^\E<\infty}=
\sum_{K\in \E}\eta^{-\zeta_{D^{\rho,K}}}
{\bf 1}_{\zeta_{D^{\rho,K}}<\infty}.
$$
Then,
$$
\EE\left(\eta^{-\zeta^\E}, \zeta^\E<\infty\right)=
\sum_{K\in \E} \EE\left(\eta^{-\zeta_{D^{\rho,K}}},
\zeta_{D^{\rho,K}} <\infty\right)<\infty.
$$
Hence, from (\ref{e40y}), we deduce
\begin{equation}
\label{e40yx}
\lim\limits_{n\to \infty}\eta^{-n}\PP(\zeta>n, Y_n\in \partial(\tT)^\E)=
\EE\left(\eta^{-\zeta^\E}, \zeta^\E<\infty\right).
\end{equation}
Then, relations (\ref{50a}), (\ref{e40y}) and (\ref{e40yx}), 
give (\ref{e32}). 

\medskip

Now, relation (\ref{50e}) is a consequence of relations 
(\ref{50f}) and (\ref{e40yx}) because they imply
\begin{eqnarray*}
\lim\limits_{n\to \infty} \eta^{-n} \PP(\zeta>n)&=&
\lim\limits_{n\to \infty} \eta^{-n} \PP(\zeta>n, Y_n\in 
\partial(\tT)^\E)\\
&=&\EE(\eta^{-\zeta^\E}, \zeta^\E<\infty)\in (0,\infty).
\end{eqnarray*}

\medskip

Let us show (\ref{50b}). First, assume $\delta$ is such that 
$\PP_\delta(\zeta^\E<\infty)>0$.
Since there is a path with
positive probability from $\delta$
to some nonempty subset of $\partial(\tT)^\E$,
a similar proof as the one showing (\ref{50e}) gives that
$$
\lim\limits_{n\to \infty} \eta^{-n}
\PP_\delta(\zeta>n)=
\EE_\delta(\eta^{-\zeta^\E},\zeta^\E<\infty)\in (0,\infty),
$$
and so the relation (\ref{50b}) is satisfied.
Now, let $\PP_\delta(\zeta^\E<\infty)=0$. Then, 
$\EE_\delta(\eta^{-\zeta^\E}, \zeta^\E<\infty)=0$
and in (\ref{50b}) we have
${\EE_\delta(\eta^{-\zeta^\E}, \zeta^\E<\infty)}/
{\EE(\eta^{-\zeta^\E}, \zeta^\E<\infty)}=0$. 
We claim that in this case we also have
$\lim\limits_{n\to \infty} \PP_\delta(\zeta>n)/\PP(\zeta>n)=0$.
In fact, $\PP_\delta(\zeta^\E<\infty)=0$ implies
\begin{eqnarray*}
(\beta_0+\theta)^{-n}\PP_\delta(\zeta>n)&=&
(\beta_0+\theta)^{-n}\PP_\delta(\zeta>n, \zeta^\E>n)\\
&=&(\beta_0+\theta)^{-n} \PP(\forall j\le n: 
Y_j\not\in (\partial(T)^\E\cup \{\D^\rho\})<\infty. 
\end{eqnarray*}
Since $\lim\limits_{n\to \infty}\eta^{-n}\PP(\zeta>n)>0$ 
and $\beta_0+\theta<\eta$, the claim follows and (\ref{50b})
is shown.

\medskip

Now, let $P^*$ be the restriction of $P$
to $\partial(\tT)^*$. The last statement we must show 
is that the vector $\varphi$ 
defined in (\ref{rev1}) is a right eigenvector of $P^*$ 
with eigenvalue $\eta$. First take $\delta\in \partial(\tT)^\E$. We 
have $\PP_\delta(\zeta^\E=0)=1$ and so
$\EE_\delta(\eta^{-\zeta^\E},\zeta^\E<\infty)=1$.
Since $P_{\delta,\delta'}>0$ and $\delta'\neq D^\rho$
imply $\delta'=\delta$, from $P_{\delta,\delta}=\eta$ we get
$$
(P^* \varphi)_\delta=
\sum_{\delta':\delta'\neq D^\rho, \delta\to \delta'} P_{\delta,\delta'}
\EE_{\delta'}(\eta^{-\zeta^\E},\zeta^\E<\infty)=\eta
=\eta\, \varphi_\delta.
$$
Now let $\delta$ be such that $\PP_\delta(\zeta^\E<\infty)=0$,
so $\varphi_\delta=0$.
Then $P_{\delta,\delta'}>0$ implies  
$\PP_{\delta'}(\zeta^\E<\infty)=0$ and so
$(P^* \varphi)_\delta=0=\eta\, \varphi_\delta$.

\medskip

Now take $\delta\not\in \partial(\tT)^\E$ with 
$\PP_\delta(\zeta^\E<\infty)>0$. Then, from
the Markov property we get,
\begin{eqnarray*}
\varphi_\delta&=&
\EE_\delta(\eta^{-\zeta^\E},\zeta^\E<\infty)=
\sum_{\delta':\delta'\neq D^\rho, \delta\to \delta'} 
\EE_\delta(\eta^{-\zeta^\E},\zeta^\E<\infty, Y_1=\delta')\\
&=&\sum_{\delta':\delta'\neq D^\rho, \delta\to \delta'} 
P_{\delta,\delta'} \; \eta^{-1}\,
\EE_{\delta'}(\eta^{-\zeta^\E},\zeta^\E<\infty)
=\eta^{-1}\, (P^* \varphi)_\delta.
\end{eqnarray*}
Hence, the result is shown. This finishes the proof of the theorem.
$\Box$

\bigskip

We will get two results from Theorem \ref{theo1}. In the first one
we supply the $Q-$process, which is the Markov chain that 
avoids some forbidden region, in our case the singleton
$\{\otimes_{L\in \D^\rho}\mu_L\}$.
In the second one we give a class of quasi-stationary
distributions, that must be compared with the irreducible case where 
there is a unique one. The $Q-$process in branching process was introduced 
in Section I.D.14 in \cite{an}. In \cite{cms} it can be found 
the construction of the $Q-$process for Markov chains 
and dynamical systems.

\medskip

In the sequel it is convenient to denote by
$\partial(\tT)^*=\partial(\tT)\setminus \{D^\rho\}$
and by $P^*$ the restriction of $P$ to $\partial(\tT)^*$.

\begin{corollary}
\label{cor1}
The following limit exists
$$
\lim\limits_{n\to \infty}\PP(Y_i=\delta_i, i=1,..,j \, | \, \zeta>n)
$$ 
for all $\delta_i\in \partial(\tT)\setminus \{D^\rho\}$, $i=1,..,k$, 
and it vanishes if some 
$\delta_i$ satisfies $\PP_{\delta_i}(\zeta^\E<\infty)=0$.

\medskip

Denote 
$$
\partial(\zeta^\E)=\{\delta\in \partial(\tT)^*:
\PP_{\delta}(\zeta^\E<\infty)>0\}.
$$ 
Then, the matrix 
$Q=\left(Q_{\delta,\delta'}: \delta,\delta'\in \partial(\zeta^\E)\right)$
given by
$$
Q_{\delta,\delta'}=\eta^{-1}\, 
P_{\delta,\delta'}\frac{\EE_{\delta'}(\eta^{\zeta^\E},
\zeta^\E<\infty)}{\EE_\delta(\eta^{\zeta^\E}, \zeta^\E<\infty)},
$$
is an stochastic matrix on $\partial(\zeta^\E)$, and it is satisfied
$$
\forall \delta_i\in \partial(\zeta^\E), i=0,..,j: \quad
\lim\limits_{n\to \infty}\PP_{\delta_0}(Y_i=\delta_i, i=1,..,j \, | \, \zeta>n)=
\prod_{i=0}^{j-1} Q_{\delta_i,\delta_{i+1}}.
$$
That is, $Q$ is the transition matrix of the Markov chain that 
never hits
$\otimes_{L\in \D^\rho}\mu_L$.
\end{corollary}

\begin{proof}
Let us prove that $Q$ is an stochastic matrix. Let
$\varphi$ be the right eigenvector of $P^*$
with eigenvalue $\eta^{-1}$  given in (\ref{rev1}). We have that
$\varphi_\delta$ vanishes when 
$\PP_{\delta}(\zeta^\E<\infty)=0$. Let $\delta\in \partial(\zeta^\E)$. 
We will use that
$P_{\delta,\delta'}=0$ if $\delta\not\rightarrow \delta'$ and that
$$
\PP_{\delta'}(\zeta^\E<\infty)=0 \hbox{ implies } 
\frac{\varphi_{\delta'}}{\varphi_\delta}=\frac{\EE_{\delta'}(\eta^{\zeta^\E},
\zeta^\E<\infty)}{\EE_\delta(\eta^{\zeta^\E}, \zeta^\E<\infty)}=0.
$$
Hence
$$
\sum_{\delta'\in \partial(\zeta^\E)}Q_{\delta,\delta'}=
\eta^{-1}\left(\sum_{\delta'\in \partial(\zeta^\E)}
P_{\delta,\delta'}\frac{\varphi_{\delta'}}{\varphi_{\delta}}\right)
=\eta^{-1} \left( \sum_{\delta'\in \partial(\tT)^*}
P_{\delta,\delta'}\frac{\varphi_{\delta'}}{\varphi_{\delta}}\right)=1.
$$
The last equality because $\varphi$ is a right eigenvector 
with eigenvalue $\eta$.
Now, from the Markov property we obtain for $n>j$,
$$
\PP(Y_i=\delta_i, i=1,..,j \, | \, \zeta>n)=
\PP(Y_i=\delta_i, i=1,..,j)\frac{\PP_{\delta_j}(\zeta>n-j)}{\PP(\zeta>n)}.
$$
Now we use the ratio limit result (\ref{50b}).
This limit vanishes if $\PP_{\delta_j}(\zeta^\E<\infty)=0$. Also it
vanishes when $\PP_{\delta_i}(\zeta^\E<\infty)=0$ for some $i<j$ because 
$P_{\delta_i,\delta_{i+1}}>0$ implies 
$\PP_{\delta_{i+1}}(\zeta^\E<\infty)=0$.

\medskip

Finally, let $\delta_i\in  \partial(\zeta^\E)$ for $i=0,..,j$. We have
\begin{eqnarray}
\nonumber
&{}&\lim\limits_{n\to \infty}\PP_{\delta_0}
(Y_i=\delta_i, i=1,..,j \, | \, \zeta>n)\\
\nonumber
&{}& =
\lim\limits_{n\to \infty}
\PP_{\delta_0}(Y_i=\delta_i, i=1,..,j)
\frac{\PP_{\delta_j}(\zeta>n-j)}
{\PP_{\delta_0}(\zeta>n)}\\
\label{rats}
&{}&=\PP_{\delta_0}(Y_i=\delta_i, i=1,..,j)
\frac{\varphi_{\delta_j}}{\varphi_{\delta_0}}\eta^{-j}\\
\nonumber
&{}& =\prod_{l=0}^{j-1} \left(\eta^{-1}P_{\delta_l, \delta_{l+1}}\,
\frac{\varphi_{\delta_{l+1}}}{\varphi_{\delta_l}}\right).
\end{eqnarray}
In (\ref{rats}) we used 
$\lim\limits_{n\to \infty}\PP(\zeta>n-j)/\PP(\zeta>n)=\eta^{-j}$,
which is a consequence of (\ref{50e}).
Then the result follows.
\end{proof}

Let $\nu=(\nu_\delta: \delta\in \partial(\tT)^*)$ be a 
probability measure on $\partial(\tT)^*$. 
If necessary, $\nu$ will be identified with its extension
on $\partial(\tT)$ with $\nu_{D^\rho}=0$. We say that $\nu$ is supported 
by some subset ${\widetilde{\partial}}\subseteq \partial(\tT)^*$ if 
$\nu({\widetilde{\partial}})=1$.
We denote by $\nu'$ the row vector associated to $\nu$.

\begin{corollary}
\label{cor2}
Every probability measure $\nu$ on $\partial(\tT)^*$ 
supported on $\partial(\tT)^\E$ satisfies 
$\nu'P^*=\eta \, \nu'$ and it is a quasi-stationary 
distribution, that is it satisfies
\begin{equation}
\label{e43}
\forall n\ge 1, \, \forall \delta\in \partial(\tT)^\E: \quad 
\PP_\nu(Y_n=\delta \, | \, \zeta>n)=\nu_\delta. 
\end{equation}
Moreover, if for some $a\le \eta$ we have
$\E(a)=\{K\in \Y(\J_\rho): \rho^K_K=a\}\neq \emptyset$,
then any probability measure ${\widetilde{\nu}}$ supported 
on $\partial(\tT)^{\E(a)}=\{\D^{\rho,K}: K\in \E(a)\}$ 
satisfies ${\widetilde{\nu}}'P^*=a\, {\widetilde{\nu}}'$ 
and it is a quasi-stationary distribution,
\begin{equation}
\label{e44}
\forall n\ge 1, \, 
\forall \delta\in \partial(\tT)^{\E(a)}: \quad 
\PP_{{\widetilde{\nu}}}(Y_n=\delta \, | \, \zeta>n)={\widetilde{\nu}}_\delta.
\end{equation}
\end{corollary}

\begin{proof} 
With the above notation and by using (\ref{eq41}) we get,
$$
(\nu' P^*)_\delta= P_{\delta,\delta} \, 
\nu_\delta=\eta \, \nu_\delta, 
$$
so $\nu' P^*=\eta \nu'$. By iteration we find 
$\nu' P^{*n}=\eta^n \, \nu'$.
Note that this is equivalent to
$$
(\nu' P^{*n})_\delta=\PP_{\nu}(Y_n=\delta)=\PP_{\nu}(\forall j\le n \; 
Y_j=\delta) =\eta^n \, \nu'_\delta.
$$ 
Now
$$
\PP_\nu(\zeta>n)=
\sum_{\delta\in \partial(\tT)^\E} (\nu' P^{*n})_\delta=\eta^n 
\left(\sum_{\delta\in \partial(\tT)^\E} \nu_\delta\right)=\eta^n.
$$ 
Hence, relation (\ref{e43}) is proven. The proof of (\ref{e44}) is 
completely similar. 
\end{proof}

This is analogous for positive eigenvectors. Let 
$\widetilde{\partial}\subseteq \partial(\tT)^\E$
be a nonempty set, then the characteristic function 
${\bf 1}_{\widetilde{\partial}}$ is a right eigenvector of $P^*$ 
with eigenvalue $\eta$. Also, if 
$\widetilde{\partial}\subseteq 
\partial(\tT)^{\E(a)}$ is a nonempty subset for some 
$a\le \eta$, then 
${\bf 1}_{\widetilde{\partial}}$ is a right eigenvector of $P^*$
with eigenvalue $a$. 
We notice that an analogous of the $Q-$process construction 
can be written on the class of states $\delta$'s that verify
$\PP_\delta(\zeta^\E<\infty))=0$ and 
$\PP_\delta(\zeta_{{\widetilde{\partial}}}<\infty)>0$,
being ${\widetilde{\partial}}=\{\delta\in \partial(\tT)^*:
P_{\delta,\delta}=\beta_0\}$.

\begin{remark}
\label{rem2a}
The results we have obtained can be easily extended to general
probability spaces. In fact, let $((X_i,\B_i): i\in I)$ 
be a finite collection of measurable spaces 
and $(\prod_{i\in I}X_i,\otimes_{i\in I} \B_i,\mu)$ be a 
probability space. For $J\in \AS$, let $\B_J=\otimes_{i\in J} 
\B_i$ be the product $\sigma$-field on $\prod_{i\in J}X_i$ 
and $\mu_J$ be the marginal on $(\prod_{i\in J}X_j, \B_J)$, 
$$
\mu_J(V)=\mu\left(V\times \,\prod_{i\in I\setminus J}X_j\right)
\hbox{ for } V\in \B_J.
$$
Then, introduce the partitions on $I$ as in Sections
\ref{sec1} and \ref{sec2} and as done in Definition \ref{def1}, 
for a probability vector $\rho=(\rho_J: J\in \AS^{(\emptyset)})$ 
define
$\Xi[\mu]=\rho_I\, \mu+ \sum_{J\in \AS} \rho_J\, \mu_J\otimes 
\mu_{J^c}$.
Then, all the results of this paper, in particular Theorem 
\ref{theo0} and Theorem \ref{theo1}, can be written in this setting.
The unique thing one must take care is to replace the expression 
$\sum\limits_{x_{K}\in \prod_{i\in K} A_i} \mu_K(x_K) 
g(x_k,x_{K^c})$ by 
$\int\limits_{\prod_{i\in K} X_i} g(x_k,x_{K^c}) d\mu_K$ 
when it is required, for instance in Lemma \ref{lemma3}.
\end{remark}

\smallskip

\begin{remark}
\label{rem3a}
The atoms of the partition $\D^\rho$ can always be
assumed to be singletons, that is $\D^\rho=\D_{si}$. 
We have not done it because on one 
hand, there is no substantial gain in notation, and on the other hand, 
this can be made only a posteriori because the 
input is the vector $\rho=(\rho_J: J\in \AS^{(\emptyset)})$ and 
the atoms are obtained once computing the set of nonempty intersections 
(\ref{e7}) from the support $\J_\rho$, the
atoms are the minimal elements, so they satisfy (\ref{e8}). 
\end{remark}

\medskip

\noindent {\bf Final comment}. 
As already said, all we have done does not require
the operation $\otimes$ to be the product
between probability measures. As it can be checked, the results can 
be extended to any operation $\otimes$ defined in the domains 
(\ref{eab0}) that satisfies commutativity, associativity 
(\ref{eab4}), stability under restriction (\ref{eab})
and $\mu_\emptyset$ is the identity element. In particular, 
 commutativity and associativity imply that for any partition
the probability measure $\otimes_{K\in \delta} \mu_K$ is well-defined. 
Obviously, the partition
$\otimes_{L\in \D^\rho} \mu_L$ could have a meaning different from 
a product measure of the marginals on the atoms, but it continue 
to play the same central role in all our constructions and results.

\bigskip

\noindent{\bf Acknowledgments}. We thank support from the CMM Basal 
CONICYT Project PB-03. We acknowledge discussions with Dr. Thierry 
Huillet from CNRS and the hospitality of 
the Laboratoire de Physique Th\'{e}orique 
et Mod\'{e}lisation at the
Universit\'{e} de Cergy-Pontoise, where this work was started.

\noindent SERVET MART\'INEZ

\noindent {\it Departamento Ingenier{\'\i}a Matem\'atica and Centro
Modelamiento Matem\'atico, Universidad de Chile,
UMI 2807 CNRS, Casilla 170-3, Correo 3, Santiago, Chile.}
e-mail: smartine@dim.uchile.cl

\label{lastpage}

\end{document}